\documentclass[12pt]{amsart}
\usepackage{amsthm}
\usepackage{amssymb}
\usepackage{amsmath}
\usepackage{graphicx, enumerate}
\usepackage{xcolor, soul}
\usepackage[all]{xy}
\usepackage{epsfig} % Required for inserting images
\usepackage[colorlinks=true,linkcolor=blue,urlcolor=blue,citecolor=blue]{
hyperref}
\usepackage[normalem]{ulem}

\def\barB{\overline{B}}
\def\d{\displaystyle}
\def\GP{\mathcal{GP}}

\def\A{\mathcal{A}}

\def\H{\mathcal{H}}
\def\M{\mathcal{M}}
\def\P{\mathcal{P}}

\def\Card{{\rm Card}}
\def\C{\mathbb{C}}
\def\D{\mathbb{D}}
\def\N{\mathbb{N}}
\def\R{\mathbb{R}}
\def\T{\mathbb{T}}

\def\acc{{\rm ac}}

\def\Id{\text{Id}}

\def\barD{\overline{\D}}

\def\Ext{\mbox{\rm Ext}}

\newtheorem{theorem}{Theorem}[section]
\newtheorem{lemma}[theorem]{Lemma}

\newtheorem{corollary}[theorem]{Corollary}
\newtheorem{proposition}[theorem]{Proposition}
\newtheorem{remark}[theorem]{Remark}
\newtheorem{example}[theorem]{Example}

\def\Nat{\mathbb{N}}

\author[V. Dimant]{Ver\'onica Dimant}
\address{Departamento de Matem\'{a}tica y Ciencias, Universidad de San
Andr\'{e}s, Vito Dumas 284, (B1644BID) Victoria, Buenos Aires,
Argentina and CONICET} \email{vero@udesa.edu.ar}

\author[S. Lassalle]{Silvia Lassalle}
\address{Departamento de Matem\'{a}tica y Ciencias, Universidad de San
Andr\'{e}s, Vito Dumas 284, (B1644BID) Victoria, Buenos Aires,
Argentina and CONICET} \email{slassalle@udesa.edu.ar}

\author[M. Maestre]{Manuel Maestre}
\address{Departamento de An\'{a}lisis Matem\'{a}tico, Universidad de Valencia, Doctor Moliner 50, 46100 Burjasot,
Valencia, Spain
}
\email{manuel.maestre@uv.es}

\thanks{The first and second authors were partially supported by CONICET PIP 11220200101609CO and ANPCyT PICT 2018-04104 and PAI-UdeSA 2023.  The third author's research was partially supported by grant PID2021-122126NB-C33 funded by MICIU/AEI/10.13039/501100011033 and by ERDF/EU and  the project GV PROMETEU/2021/070.}


\begin{document}
\begin{abstract} 
In the early nineties, R. M.  Aron, B. Cole, T. Gamelin and W.B. Johnson initiated the study of the maximal ideal space (spectrum) of Banach algebras of holomorphic functions defined on the open unit ball of an infinite dimensional complex Banach space. Within this framework, we investigate the fibers and  Gleason parts of the spectrum of the algebra of holomorphic and uniformly continuous functions on the unit ball of $\ell_p$ ($1\le p<\infty$).
We show that the inherent geometry of these spaces provides a fundamental ingredient for our results. We prove that whenever $p\in\N$ ($p\ge 2$), the fiber of every $z\in B_{\ell_p}$ contains a set of cardinal $2^{\mathfrak c}$ such that any two elements of this set belong to different Gleason parts. For the case $p=1$, we complete the known description of the fibers, showing that, for each $z\in\overline B_{\ell_1''}\setminus S_{\ell_1}$, the fiber over $z$ is not a singleton. Also, we establish that different fibers over elements in $S_{\ell_1''}$ cannot share Gleason parts. 
\end{abstract}

\baselineskip=.65cm

\title{Fibers and Gleason parts for the maximal ideal space of  $\A_u(B_{\ell_p})$}

\maketitle

\section{Introduction} 
Over the years, for the study of the maximal ideal space (or spectrum) of an algebra of holomorphic functions, two partitions of this set have proved useful in describing its analytical structure. On the one hand, the notion of \textit{part}, introduced by Gleason \cite{gleason1957} in the 1950's, comes from an equivalence relation arising from the metric of the surrounding space. On the other hand, the idea of \textit{fiber}, developed since the foundational work of Schark \cite{schark1961maximal}, refers to the restriction of homomorphisms to the dual of the underlying space. These two approaches to analyzing the structure of the spectrum have followed separate paths, with some meeting points in certain works. The main goal of this article is to contribute to the description and interrelation between fibers and  Gleason parts of the spectrum of the algebra of holomorphic and uniformly continuous functions on the unit ball of $\ell_p$ ($1\le p<\infty$).

In the context of several complex variables, the study of Gleason parts was carried on the maximal ideal space of $\H^\infty(B)$ (the algebra of bounded holomorphic functions on $B$), where $B$ is the open unit disk in $\C$, the polydisk or the Euclidean ball in $\C^n$. There is extensive literature on this research, for instance,  \cite{bear2006lectures,garnett2006bounded,gorkin1989gleason,hoffman1967bounded,konig1969gleason,mortini1994gleason,stout1971theory}. When restricting to $\A(B)$, the algebra of holomorphic functions on $B$ which are continuous on its closure $\overline B$, the spectrum  $\M(\A(B))=\overline{B}$ and its Gleason parts are easily described as $B$  and $\{\delta_z\}$ for every $z$ in the boundary of $B$. The situation is far more involved in the infinite dimensional framework. Whenever $B_X$ is the open unit ball of an infinite dimensional Banach space $X$, the study of  $\H^\infty(B_X)$ and of some of its closed subalgebras, as $\A_u(B_X)$ (uniformly continuous holomorphic functions on $B_X$), was initiated in 1991 with the works of Aron, Cole, Gamelin, Johnson and others, starting with the seminal paper  \cite{aron1991spectra}. In this setting most of the work was addressed to study the size of the fibers, the Shilov and Choquet boundaries, the cluster theorem, and the containment of big analytic structures inside the fibers \cite{acosta2009boundaries,acosta2007shilov,arenson1983gleason,aron2016cluster,aron2012cluster,aron2018analytic,carando2023homomorphisms,castillo2023polynomial,choi2021boundaries, cole1992analytic,farmer1998fibers,globevnik1978interpolation,johnson2014cluster,johnson2015cluster}. It is interesting to mention that if there is a continuous polynomial on $X$ which is not weakly continuous on bounded sets, then the fibers of the spectrum of $\A_u(B_X)$  over any point in $B_{X''}$ are very rich, thus revealing a relevant difference with the finite-dimensional case. 

To continue this introduction, let us establish some notation and definitions. Consider an infinite dimensional complex Banach space $X$ with open unit ball $B_X$, unit sphere $S_X$, closed unit ball $\barB_X$, and topological dual and bidual spaces $X'$ and $X''$, respectively. Inside $\H^\infty(B_X)$ (the algebra of all bounded holomorphic mappings on $B_X$) we are interested in $\A_u(B_{X})$, the algebra of holomorphic functions $f\colon B_{X}\to \C$ that are uniformly continuous with respect to the norm of $X$. Additionally, $\A_a(B_X)$ stands for the closed subalgebra of the latter space generated by $X'$ and the constant functions. Thus, the inclusions $\A_a(B_{X}) \subset \A_u(B_{X}) \subset \H^\infty(B_X)$ hold, and the three function spaces are complex Banach algebras endowed with the supremum norm, $\|f\|=\sup\{|f(x)|\colon \|x\|<1\}$.  

For a commutative complex Banach algebra $\A$,  its maximal ideal space (or spectrum for short) $\M(\A)$  is the set of all nonzero $\C$-valued homomorphisms.
As each function in $\A_u(B_{X})$ extends continuously to $S_{X}$, then
$\M(\A_u(B_{X}))$ contains the point evaluations $\delta_x$ for all $x \in X,  \|x\|\le 1$.
Furthermore, according to \cite{davie1989theorem}, every $f\in \H^\infty(B_X)$ extends (via the extension developed in \cite{aron1978hahn}) to $\tilde f\in \H^\infty(B_{X''})$ in a natural way.  We refer to this mapping  $f \leadsto \tilde f$ as the canonical extension, that constitutes a Banach algebra homomorphism.  Using standard arguments, it can be seen that this extension maps functions in $\A_u(B_X)$ to functions in $\A_u(B_{X''})$. Thus, each point $z \in B_{X''}$ (or $\barB_{X''}$) corresponds to an element $\tilde{\delta}_{z}\in \M(\H^\infty(B_X))$ (or $\M(\A_u(B_{X})))$, where $\tilde{\delta}_{z}(f)=\tilde f(z)$. As $X'$ is included in $\A_u(B_X)$, $\tilde{\delta}_{z} \ne \tilde{\delta}_{w}$ whenever $z\ne w$ belong to $B_{X''}$ (or to $\barB_{X''}$ for $\M(\A_u(B_{X}))$).
To simplify the notation, we write $\delta_{z}(f)$ instead of $\tilde \delta_{z}(f)$ for any $z\in B_{X''}$ and $f\in \H^\infty(B_X)$ and also for any $z\in S_{X''}$ and $f\in \A_u(B_{X})$, as in this case, $\tilde f(z)$ is uniquely determined. 

In general, for infinite-dimensional Banach spaces, it is well-known (see, e.g., \cite{aron1991spectra}) that $\M(\A_u(B_X))$ usually contains much more than mere evaluations at points of $\barB_{X''}$. However, there are some distinguished cases, such as when $X = c_0$ (the Banach space of complex null sequences), in which $\M(\A_u(B_{c_0}))=\{\delta_z\colon z\in \overline{B}_{\ell_\infty}\}$. This result, which could be traced to \cite{dineen2012complex}, (see also \cite[Introduction to Section 2]{aron2020gleason}), is a particular case of a Corona theorem. Loosely speaking, we say that a Corona theorem holds when the evaluation homomorphisms form a dense set in the spectrum of the considered algebra. It was Carleson \cite{carleson1962interpolations}, who set the first Corona theorem for $\H^\infty(\D)$, with $\D$ the complex unit disk. For finite-dimensional Banach spaces, the Corona theorem trivially holds for $\A_u(B_X)$, but it is an open question for general infinite dimensional ones, even for the classical $\ell_p$-spaces, $1\le p <\infty$.   \\
To achieve a better understanding of the extent of the evaluation morphisms in the spectrum, weaker structures can be used. That is the case of the cluster set $Cl_{B_X}(f,z)$ at $z\in \barB_{X''}$, defined as the set of all limits of values of $f$ along nets in $B_X$ converging weak-star to $z$. A {\it Cluster value theorem} at $z \in \bar B_{X''}$ means the  assertion
\begin{equation*}\label{cluster thm}
Cl_{B_X}(f,z) = \widehat f(\M_z(\A_u(B_{X}))), \qquad \forall f\in \A_u(B_{X}), 
\end{equation*}
where $\widehat f$ is the Gelfand transform of $f$ and $\M_z$ is the fiber in $\M(\A_u(B_{X}))$ over $z$ (see Section~\ref{Sec:preliminaries} for the definitions). It is proved in  \cite[Thm. 3.1]{aron2012cluster} that $\A_u(B_{\ell_p})$ satisfies the Cluster value theorem at $z=0$ for every $1< p<+\infty$  and it is shown in   \cite[Thm. 4.1]{aron2012cluster}  that $\A_u(B_{\ell_2})$ satisfies the Cluster value theorem at any point of the unit ball. 
For more information about the Corona and Cluster theorems, we refer the reader to the survey \cite{CaGaMaSeSpectra}.

 We aim to study the structure of the fibers and Gleason parts (see precise definitions below) of $\M(\A_u(B_{\ell_p}))$ for $1\leq p<\infty$. This paper represents a natural extension of the work done by several researchers in \cite{aron2012cluster,aron2016cluster,aron2020gleason, aron2018analytic,farmer1998fibers,johnson2014cluster,johnson2015cluster}. In particular, this kind of comprehensive examination of fibers,  Gleason parts, and their interaction, for the spectra of  $\A_u(B_{c_0})$ and  $\H^\infty(B_{c_0})$, has been conducted in \cite{aron2020gleason}. However, it is important to note that functions in  $\A_u(B_{\ell_p})$  display notably more complex behavior than those in  $\A_u(B_{c_0})$. For instance, for $k\ge p$, the polynomial $P(x)=\sum_j x_j^k$ belongs to $\A_u(B_{\ell_p})$ but it is not weakly continuous on bounded sets while functions in $\A_u(B_{c_0})$ are all approximable by finite type polynomials, implying that $\A_u(B_{c_0})$ coincides with $\A_a(B_{c_0})$ (see \cite[Prop.~1.59]{dineen2012complex} or \cite[Sect.~3.4]{gamelin1994analytic}). As a consequence of the latter fact, fibers in $\M(\A_u(B_{c_0}))$ are singletons and each fiber is identified with an element of $\barB_{\ell_\infty}$. For any $1<p<\infty$, the fibers (of the spectrum of $\A_u(B_{\ell_p})$) over elements in $S_{\ell_p}$ are also singletons while the fibers over elements in $B_{\ell_p}$ have a very rich structure, see \cite[Prop.~2.6]{aron2016cluster}. Finally, the spectrum of $\A_u(B_{\ell_1})$ deserves special attention.
The fiber over any element $z$ in $S_{\ell_1}$ is just $\{\delta_z\}$ while the fiber over any $z\in B_{\ell_1''}$ has uncountable elements. Moreover, there are infinitely many points in $S_{\ell_1''}$ whose fibers have at least the cardinality of the continuum $\mathfrak{c}$, see \cite[Props.~3.1,~3.2 and 3.3 ]{aron2016cluster}. However, the description of the fibers of $\A_u(B_{\ell_1})$ was not complete until now that we distinguished all fibers into these two groups, according to geometric properties of $z$.

To be more precise in our objective, we concentrate on the interaction between fibers and Gleason parts of the maximal ideal space of  $\A_u(B_{\ell_p})$ with special emphasis on its non-corona part; i.e. on the closure in $\M(\A_u(B_{\ell_p}))$ of the set of all homomorphisms that are evaluations on points of $\barB_{\ell_p}$ (for $1< p < \infty$) and $\barB_{\ell_1''}$ (for $p=1$).
\\
Section~\ref{Sec:preliminaries} is devoted to recalling some preliminaries and presenting the guided lines of our study and some general results that will be useful to us. The connection with the complex geometry of the Banach space $X$ will be discussed, since, e.g., 
in Proposition~\ref{prop:GP_complex_extrem} it is shown that $z\in S_{X''}$ is a complex extreme point of $\overline{B}_{X''}$ if and only if the Gleason part of $\delta_z$ is contained in the fiber over $z$, $\M_z$ (see \eqref{eq:fibra en z} for the definition).

In Section 3 we study the Gleason parts of $\M(\A_u(B_{\ell_p}))$, $1<p<\infty$. We show in Theorem~\ref{BetaNenM0} that fiber $\M_0$ in $\M(\A_u(B_{\ell_p}))$ contains a set of cardinal $2^{\mathfrak c}$ such that any two elements of this set belong to different Gleason parts. This result is extended in Proposition~\ref{prop: p entero}, to any fiber in $x\in B_{\ell_p}$  for the case $ p \in \N$.
\\
Section 4 deals with the structure of $\M(\A_u(B_{\ell_1}))$. We present a new approach to describing the $z\in S_{\ell_1''}$ with no singleton fiber that allows us to identify these elements fully. Additionally, as a consequence of Proposition~\ref{prop:GP_complex_extrem}, we show that no Gleason part can contain homomorphisms from more than one {\sl edge} fiber.
\\
Finally, in Section 5  we propose some open problems and add some comments about $\M(\A_u(B_{\ell_p}))$ ($1\leq p<\infty$). In addition, as an appendix, we include a description of the complex extreme points of the unit ball of a $p$-sum of two Banach spaces, $X \oplus_p Y$, as we consider it has interest on its own.

%%%%%%%%%%%%%%%%%%%%%%%%%%%%%%%%%%%%%%%%%%%%%%%%%%%%%%
\section{Preliminaries and general results}
\label{Sec:preliminaries}
%%%%%%%%%%%%%%%%%%%%%%%%%%%%%%%%%%%%%%%%%%%%%%%%%%%%%%

For any commutative Banach algebra  $\A$ with identity, its spectrum  $\M(\A)$ considered with the $\sigma(\A^*,\A)$ topology ($w^*$ for short) is a compact Hausdorff space. Via the Gelfand transform $f\mapsto \hat f$, $\A$ is identified with a subspace of the space of continuous functions of $C(\M(\A))$, where $\hat f(\varphi)=\varphi(f)$ for any $f\in \A$ and any $\varphi\in \M(\A)$.

Either for $\A=\A_u(B_X)$ or $\A=\H^\infty(B_X)$, the mapping $\pi\colon\M(\A) \to \barB_{X''}$ given by $\pi(\varphi):= \varphi|_{X'}$, is a well-defined linear operator (as $X' \subset \A$) and is surjective \cite{aron1991spectra}. With the natural embedding $J_X\colon X\to X''$, we have the following commutative diagram: 
\begin{equation*}\label{diagrama}
\xymatrix{
B_X \ar[r]^\delta \ar[rd]_{J_X}  & \M(\A)\ar[d]^{\pi}\\
& \barB_{X''}.\\
}
\end{equation*}
Conventionally, omitting the algebra when it is clear, the {\em fiber of $\M(\A)$ over $z\in \barB_{X''}$,}  will be denoted as

\begin{equation}\label{eq:fibra en z}
\M_z:= \M_z(\A)= \{ \varphi \in \M(\A) \ | \
\pi(\varphi) = z\}.
\end{equation}

Understanding the structure of these fibers is beneficial in the examination of Gleason parts that are defined as follows. First, we recall the Gleason distance between $\varphi$ and $\psi$ in $\M(\A)$, given by $\|\varphi - \psi\|:=\sup \{|\varphi(f) - \psi(f)| \colon \|f\| \leq 1\}$. Clearly, $\|\varphi - \psi\|\le 2$ and it can be seen, with some work, that the relation $\varphi \sim \psi$ iff $\|\varphi - \psi\|< 2$ is an equivalence relation in $\M(\A)$, see for instance \cite[Lem.~16.1]{stout1971theory}. The sets of equivalence classes are referred to as the Gleason parts of $\M(\A)$. We write
\begin{equation}\label{Def-Gleason1}
\GP(\varphi):= \{\psi\in\M(\A)\colon\  \|\varphi - \psi\|< 2\}
\end{equation}
for the Gleason part containing $\varphi$. 
A Gleason part with a unique element is called trivial or a \textit{singleton}.

When trying to figure out whether two elements of $\M(\A)$ belong or not to the same Gleason part, the pseudo-hyperbolic distance comes into play. Recall that the {\em pseudo-hyperbolic distance} between $\varphi$ and $\psi$ in $\M(\A)$ is
$$
\rho(\varphi, \psi) :=
\sup\{|\varphi(f)|\colon\  f \in \mathcal A, \|f\| \leq 1, \psi(f) = 0 \}.
$$
For the uni-dimensional unit disc $\D$, when $\A = \A(\D)$ or $\A=\H^\infty(\D)$, the pseudo-hyperbolic distance for $\lambda$ and $\mu$ in $\D$ is give by
$$
\rho(\delta_\lambda, \delta_\mu) = \Big| \frac{\lambda - \mu}{1 - \overline{\lambda}\mu}\Big|.
$$
The formula given above remains true if $\A=\A(\D)$ for $\lambda, \mu\in \overline{\D}$, if $|\lambda|=1$ and $\lambda \ne \mu$, in this case, $\rho(\delta_\lambda, \delta_\mu) =1$. Also, it extends coordinatewise to $\A_u(B_{c_0})$ and $\H^\infty(B_{c_0})$ \cite[Example~1.7]{aron2020gleason}. See Examples 1.8 and 1.9 in \cite{aron2020gleason} for the pseudo-hyperbolic distance between $\delta_x$ and $\delta_y$ in $\A_u(B_{X})$ and $\H^\infty(B_{X})$, for $X=\ell_2$ and $X=\mathcal L(H)$ where $\mathcal L(H)$ is the Banach space of bounded linear operators from a Hilbert space $H$ into itself. 
 
It follows from \cite[Thm.~2.8]{bear2006lectures} that $ \|\varphi - \psi\|< 2$ iff $\rho(\varphi, \psi) < 1$, thus for the Gleason part containing $\varphi$, we will use \eqref{Def-Gleason1} or its equivalent form 
\begin{equation}\label{Def-Gleason2}
\GP(\varphi)=\{\psi\in\M(\A)\colon \rho(\varphi, \psi) < 1\}.
\end{equation}

Now, we are in a position to establish some basic results about fibers and Gleason parts that will be used throughout this work. 
The following result is another way to present \cite[Thm.~3.9]{stout1971theory}. We give a proof for the sake of completeness. 

\begin{lemma}\label{Lema-util}
Let $X$ be a complex Banach space, $\A=\A_u(B_X)$ or $\H^\infty(B_X)$ and let $\varphi, \psi\in \M(\A)$. Suppose that there is a net  $(f_\beta)_\beta\subset \A$ such that $\|f_\beta\|\le 1$ for all $\beta$, $\varphi(f_\beta)\to \lambda_0$ for some $\lambda_0\in \barD \setminus\{1\}$ and $\psi(f_\beta)\to 1$. Then $\GP(\varphi)\not=\GP(\psi)$.
\end{lemma}

\begin{proof} For $0<\alpha <1$ consider the Mo\"ebius function $\xi_\alpha\colon \barD\to \barD$,
$\xi_\alpha(\lambda)= \d\frac{\alpha - \lambda}{1 -  \alpha \lambda}$. As $\xi_\alpha \in \A(\D)$, $\xi_\alpha\circ f\in \A$ for all $\alpha$ and all $f\in \A$, $\|f\|\le 1$.   Since, $\xi_\alpha(\lambda)= (\alpha - \lambda)\sum_{m=0}^\infty (\alpha \lambda)^m$ and  this series converges uniformly on $\overline{\D}$, we see that  $\phi(\xi_\alpha\circ f)=
\xi_\alpha(\phi(f))$, for every $\phi\in \M(\A)$ and $f\in\A$, $\|f\|\le 1$. Hence,
$$
\varphi(\xi_\alpha\circ f_\beta)=\xi_\alpha(\varphi(f_\beta))\underset{\beta}{\longrightarrow}\xi_\alpha(\lambda_0)\quad  \text{and}\quad \xi_\alpha(\lambda_0) \underset{\alpha \to 1^-}{\longrightarrow} 1.
$$ 
On the other hand, 
$$\psi(\xi_\alpha\circ f_\beta)=\xi_\alpha(\psi(f_\beta))\underset{\beta}{\longrightarrow}\xi_\alpha(1)=-1.
$$
As $\|\xi_\alpha\circ f_\beta\|\le 1$ for every $\alpha$ and $\beta$, we see that 
$$\|\varphi-\psi\| \ge | \varphi(\xi_\alpha\circ f_\beta)- \psi(\xi_\alpha\circ f_\beta)|,$$
and the result follows.
\end{proof}

\begin{remark} \label{ADLM basicos}
The following two facts about Gleason parts and their relation with fibers for $\A=\A_u(B_X)$ or $\H^\infty(B_X)$ appeared in \cite[Prop.~1.1]{aron2020gleason}:
\begin{enumerate}[\upshape (a)]
   \item \label{Rem.BX'' en GP0} $\{\delta_z\colon z\in B_{X''}\} \subset \GP(\delta_0)$.
    \item \label{Rem.borde-interior separados} If $z\in S_{X''}$, $w\in B_{X''}$, $\varphi\in \M_z$,  $\psi\in \M_w$, then $\GP(\varphi)\not=\GP(\psi)$.
\end{enumerate}
\end{remark}

For an infinite dimensional Banach space $X$, the fiber $\M_{z}$ of $\M(\H^\infty(B_X))$ over any $z\in\barB_{X''}$ is infinite \cite[Thm.~11.1]{aron1991spectra}. If, in addition, $X$ admits a continuous polynomial whose restriction to the open unit ball is not weakly continuous at $0$, the following facts hold: the fiber $\M_z$ of $\M(\A_u(B_X))$ is infinite for any $z\in B_{X''}$ \cite[Prop. 2.6]{aron2016cluster} and $\{\delta_z\colon z\in B_{X''}\}\subsetneqq \GP(\delta_0)$ \cite[Cor.~1.3]{aron2020gleason}. Item~\eqref{Rem.BX'' en GP0} of the above remark, says that $\GP(\delta_0)$ meets all the fibers over elements $z\in  B_{X''}$ and, the latter result \cite[Cor.~1.3]{aron2020gleason} says that there are situations for which $\GP(\delta_0)$ may contain several elements of the same fiber. At the same time, item~\eqref{Rem.borde-interior separados}, shows that $\GP(\delta_0)$ does not meet any $\M_z$ for  $z\in S_{X''}$. In \cite[Thm.~3.7]{aron2020gleason} it is shown that every fiber in $\M(\H^\infty(B_{c_0}))$ over a point in $B_{\ell_\infty}$ contains $2^c$ discs lying in different Gleason parts. For this space, $c_0$, the case of $\M(\A_u(B_{c_0}))$ is very different, starting from the fact that all the fibers are singletons. Also, a complete characterization of its Gleason parts is given. For instance, in \cite[Cor.~2.6]{aron2020gleason}, it is shown that for all $z$ in the distinguished boundary $\T^\N$, $\GP(\delta_z)=\{\delta_z\}$. However, there are many other points $z\in S_{\ell_\infty}$ for which the identifications $\GP(\delta_z) \thickapprox \D^{n}$ (a finite dimensional polydisk) and $\GP(\delta_z) \thickapprox B_{\ell_\infty}$ are satisfied.  
\\
This type of behavior leads us to face different natural questions about fibers, Gleason parts, and their interaction. Below we list some of them.

\begin{itemize}
    \item Which fibers or Gleason parts are nontrivial? That is, not a singleton.
    \item Which fibers intersect several Gleason parts?
    \item Which fibers contain singleton Gleason parts?
    \item When is there a copy of a ball (or a disk) at the intersection of a fiber and a Gleason part?
    \item Which Gleason parts intersect different fibers?
\end{itemize}

Addressing these questions poses a challenging task, even if we concentrate on $\A_u(B_X)$ for $X=\ell_p$, $1\le p <\infty$. Next, we will present some general results aimed at providing tools to make progress on some of the aforementioned issues.

The following proposition provides us with sufficient conditions to ensure that two homomorphisms that are in the non-corona part of the spectrum of the algebra belong to different Gleason parts. Recall that for $f\in \A_u(B_X)$ we denote by $\widetilde f\in\A_u(B_{X''})$ its canonical extension.

\begin{proposition}\label{tecnical-prop}
  Let $X$ be a complex Banach space and   let $\varphi\not=\psi\in\M(\A_u(B_X))$ such that $\varphi=w^*-\lim \delta_{z_\alpha}$, $\psi=w^*-\lim \delta_{w_\alpha}$ with $(z_\alpha), (w_\alpha)\subset \overline B_{X''}$.
    \begin{enumerate}[\upshape(a)]
    \item \label{lem:diff GP-1} If there exists $T\in\mathcal L(X,X)$ such that $\|T\|\le 1$, $T''(z_\alpha)=z_\alpha$, $T''(w_\alpha)=0$ for all $\alpha$ and $\varphi\not\in \GP(\delta_0)$ then $\GP(\varphi)\not=\GP(\psi)$.
    \item \label{lem:diff GP-2} If there exists $P\in\mathcal P(^mX)$, $\|P\|= 1$, such that $\widetilde P(z_\alpha)\to 1$ and $w_\alpha=\lambda_\alpha z_\alpha$ with $\lambda_\alpha^m\not\to 1$ then $\GP(\varphi)\not=\GP(\psi)$.
    \end{enumerate}
\end{proposition}

\begin{proof} To prove \eqref{lem:diff GP-1}, note that for any $f\in \A_u(B_X)$ and $T$ as in the statement, we have
$$
\begin{array}{rl}
\varphi(f\circ T)&= \lim\limits_\alpha \delta_{z_\alpha}(f\circ T)=\lim\limits_\alpha\widetilde{f\circ T}(z_\alpha)=\lim\limits_\alpha \widetilde f(T''(z_\alpha))\\
& =\lim\limits_\alpha \widetilde f(z_\alpha)=\varphi(f)\quad \textrm{and}  \\
\psi(f\circ T)&= \lim\limits_\alpha \delta_{w_\alpha}(f\circ T)=\lim\limits_\alpha\widetilde{f\circ T}(w_\alpha)=\lim\limits_\alpha \widetilde f(T''(w_\alpha))=f(0).
\end{array}
$$
As $\varphi\not\in \GP(\delta_0)$, there exists a net $(f_\beta)\subset \A_u(B_X)$ such that $\|f_\beta\|\le 1$, $f_\beta(0)=0$ and $\varphi(f_\beta)\to 1$. Now, $(f_\beta\circ T)\subset \A_u(B_X)$ with $\|f_\beta\circ T\|\le 1$, $\varphi(f_\beta\circ T)=\varphi(f_\beta)\to 1$ and $\psi(f_\beta\circ T)=f_\beta(0)=0$, showing that $\varphi$ and $\psi$ belong to different Gleason parts.

For \eqref{lem:diff GP-2}, with $P$ as in the statement, observe that 
$$
\varphi(P)=\lim_\alpha \delta_{z_\alpha}(P)=\lim 
\widetilde{P}(z_\alpha)= 1.
$$
On the other hand, taking a subnet if necessary, we may consider that $\lambda_\alpha^m\to \lambda   \not= 1$. Hence, 
$$
\lim_\alpha \delta_{w_\alpha}(P)=\lim_\alpha \widetilde{P}(\lambda_\alpha z_\alpha)=\lim_\alpha \lambda_\alpha^m\widetilde{P}(z_\alpha)=\lambda.
$$
Now, an application of Lemma~\ref{Lema-util} concludes the proof.
\end{proof}
    
\begin{remark}
The Gleason metric on the spectrum of $\A_u(B_X)$ is determined by  the normed space of all continuous polynomials $\P(X)$. That is, 
if $\varphi, \psi\in\M(\A_u(B_X))$ then
\[
\|\varphi-\psi\|=\sup\{|\varphi(P)-\psi(P)|\colon\, P\in\P(X),\, \|P\|\le 1\}.
\]
\rm

Indeed,  if $f\in \A_u(B_X)$, the Taylor series expansion converges uniformly to $f$ on $rB_X$ for any $0<r<1$. Then  we can proceed as in \cite[Prop.~5.2]{mujica1991linearization} to derive that the sequence of Cesaro means of the Taylor series expansion of $f$ at 0, say $(Q_n)$, converges uniformly on $rB_X$ to $f$ for all $0<r<1$, and $\|Q_n\|\leq \|f\|$ for every $n$. On the other hand, by using the uniform continuity of $f$, the family $(f_r)$ with $f_r(x):=f(rx)$ converges uniformly to $f$ on $B_X$. Additionally, for each $r$, the sequence $(Q_{n,r})_n$ defined by $Q_{n,r}(x):=Q_n(rx)$ converges uniformly to $f_r$ on $B_X$ with
$$
\|Q_{n,r}\|\leq \|Q_n\|\leq \|f\|.
$$
Thus, for every $\varphi, \psi\in\M(\A_u(B_X))$ we have 
\begin{eqnarray*}
\|\varphi-\psi\| &= &\sup\{|(\varphi-\psi)(f)|\colon\ f\in {\A_u(B_X)},\ 
\|f\|\leq 1\}\\
&=&\sup\{|(\varphi-\psi)(P)|\colon\ P\in \P(X), \ 
\|P\|\leq 1\},  
\end{eqnarray*}
which shows the statement.
\end{remark}

In the study of the fibers or Gleason parts, certain elements in $\M(\A)$ possess geometric properties that make them distinctive, as they constitute singletons. By \cite[Pag. 162]{stout1971theory}, this is the case of the strong boundary points and peak points that we define next. A point $\varphi \in \M(\A)$ is a \textit{strong boundary point} for $\A$ if for every open neighborhood $U\subset \M(\A)$ of $\varphi$ there is $f\in \A$ so that $\|f\| = \hat f(\varphi)=1$ and $|\hat f(\psi)| <1$ if $\psi\not\in U$. Also, $\varphi$ is a \textit{peak point} for $\A$ if there is $f\in\A$ such that $\hat f(\varphi)=1$ and $|\hat f(\psi)| <1$ for all $\psi \in \M(\A)\setminus \varphi$. 

For a function algebra $\A$ contained in $C_b(\Omega)$, the space of complex bounded continuous functions on $\Omega$, a topological Hausdorff space, we may consider the class of elements described above for the particular case of points in $\Omega$ (rather than for homomorphisms in $\M(\A)$). This concept appeared related to the study of a generalization of the notion of boundary due to Globevnik \cite{globevnik1978interpolation, globevnik1979boundaries}. 
A point $x\in \Omega$ is a \textit{strong boundary point} for $\A$ (in the sense of Globevnik) if for each open neighborhood $U\subset \Omega$ of $x$, there exists $f\in \mathcal{A}$ such that $|f(x)|=1$ and $\sup_{y\in \Omega\setminus U}|f(y)|<1$. On the other hand, $x$ is a \textit{strong peak point} for $\A$  (in the sense of Globevnik) if there exists $f\in \mathcal{A}$ satisfying $|f(x)|=1$ and $\sup_{y\in \Omega\setminus U}|f(y)|<1$ for all open neighborhoods $U\subset \Omega$ of $x$.
\\
The setting makes clear which definition is used. When we talk about $\varphi\in \M(\A)$ we use the classical definitions (even for evaluations $\delta_x$) and for points $x\in \Omega$ we use the notions in the sense of Globevnik.

The classes of extreme points cannot be avoided in studying geometric properties. For $C\subset X$ a convex set, a point $z$ in $C$ is a \textit{real extreme point} of $C$ if $z$ is not the midpoint of any line segment contained in  $C$. Equivalently, if $z+t y\in C$ for all $t\in\R$, $|t|\le 1$, then $y=0$. Also, $z$ in $C$ is a \textit{complex extreme point} of $C$ if $z+\lambda y\in C$ for all $\lambda\in\barD$ implies that $y=0$. The real extreme points of $C$ are denoted by $\Ext_{\R}(C)$ while the complex extreme points are denoted by $\Ext_{\C}(C)$. Clearly we have $\Ext_{\R}(C)\subset \Ext_{\C}(C)$.

When considering the intersection between fibers and Gleason parts, complex extreme points of the unit ball of $X''$ are relevant.
For the uniform algebra on $\barB_X$ generated by $X'$, $\A_a(B_X)$, it is easy to see that $\M_z(\A_a(B_X))=\{\delta_z\}$, for every $z\in \barB_{X''}$. Therefore, the spectrum of $\A_a(B_X)$ is identified with $\barB_{X''}$, and the fibers are singletons. On the other hand, a result of Arenson, \cite[Thm.]{arenson1983gleason}, shows that the strong boundary points for this algebra coincide with the set $\Ext_{\C}(\barB_{X''})$.
Rewriting the statement of \cite[Thm.]{arenson1983gleason} in our words, Arenson's result applies to the closed algebra generated by functions in $\A_u(B_{X''})$ which are canonical extensions of functions in $\A_a(B_{X})$.  There, it is proved that the strong boundary points for this algebra coincide with the set $\Ext_{\C}(\barB_{X''})$. That is, $z\in \Ext_{\C}(\barB_{X''})$ if and only if for each open $w^*$-neighborhood $U$ of $z$ in $\barB_{X''}$ there exists $f\in\A_a(B_{X})$ such that $\|f\|=\tilde f(z)=1$ and $|\tilde f(w)|< 1$ for all $w\in \barB_{X''}\setminus U$. As a consequence, for $\A_u(B_{X})$, if $w\not= z\in S_{X''}$ then $\GP(\delta_w)\not = \GP(\delta_z)$. Moreover,  in this case $\GP(\varphi)\not = \GP(\psi)$ for any $\varphi\in\M_z, \psi\in \M_w$. On the other hand, if $z\in S_{X''}$ such that $z\not\in \Ext_{\C}(\barB_{X''})$ then it should exist $u\not= 0 \in X''$ such that $\|z+\lambda u\|=1$ for all $\lambda\in\overline\D$. Thus, the mapping 
\begin{eqnarray*}
    \D &\to &\M(\A_u(B_X))\\
    \lambda &\mapsto &\delta_{z+\lambda u}
\end{eqnarray*} is analytic. Since the image of an open convex set through an analytic injection is contained in a single Gleason part (\cite[Lem.~2.1]{hoffman1967bounded} or \cite[Prop.~3.4]{aron2020gleason}), we obtain that $\{\delta_{z+\lambda u}:\, \lambda\in\D\}\subset \GP(\delta_z)$. We state these facts in the following proposition.

\begin{proposition} \label{prop:GP_complex_extrem}
Let $X$ be a complex Banach space and $z\in S_{X''}$. For the spectrum of $\A_u(B_X)$ we have
\begin{enumerate}[\upshape(a)]
    \item If $z\in \Ext_{\C}(\barB_{X''})$ then $\GP(\delta_z)$ cannot intersect any other fiber. Indeed,
    for any $w\not= z\in S_{X''}$, $\varphi\in\M_z, \psi\in \M_w$ it holds $\GP(\varphi)\not = \GP(\psi)$.

    \item If $z\not\in \Ext_{\C}(\barB_{X''})$ then  $\GP(\delta_z)$ intersects a disk of fibers: there is $u\not= 0 \in X''$ such that $\{\delta_{z+\lambda u}:\, \lambda\in\D\}\subset \GP(\delta_z)$.
\end{enumerate}

\end{proposition}

%%%%%%%%%%%%%%%%%%%%%%%%%%%%%%%%%%%%%%%%%%%%%%%%%%%%%%
\section{Fibers and Gleason parts of $\M(\A_u(B_{\ell_p}))$, $1<p<\infty$.}
\label{Sec:elep}
%%%%%%%%%%%%%%%%%%%%%%%%%%%%%%%%%%%%%%%%%%%%%%%%%%%%%%

An important example in our context arises when $X$ is a complex uniformly convex space.  In this case it is proved in  \cite[Prop.~4.1]{farmer1998fibers}   that $z$ is a strong peak point in the sense of Globevnik for $\A_a(B_X)$. Now, applying \cite[Thm. 3.3 (2)]{choi2021boundaries} we get:

\begin{proposition}\label{Prop-fibers-GPs1} 
Let $X$ be a uniformly convex complex Banach space. Then, for any $z\in S_X$,  the evaluation $\delta_z$ is a peak point for $\A_u(B_X)$ and  
$$ 
\M_z=\{\delta_z\}=\GP(\delta_z).
$$
\end{proposition} 

In particular, the above result applies to $X=\ell_p$ for $1<p<\infty$. 

The next result summarizes what else is known about fibers and Gleason parts of $\M(\A_u(B_{\ell_p}))$, for $1<p<\infty$. From \cite[Thm.~3.1]{ aron2018analytic}, we know that  for any $z\in B_{\ell_p}$, the complex disk $\D$ can be analytically injected in $\M_{z}$. Moreover, the construction of the copy of such a disk contains the evaluation $\delta_z$. By Remark~\ref{ADLM basicos}, since $\delta_z$ and $\delta_0$ belong to the same Gleason part, we derive item \eqref{Disco en Mz y GP0} below. The same reasoning yields item \eqref{Bola en M0 y GP0} below when applied to \cite[Prop.~3.10]{ aron2018analytic}.

\begin{proposition}[\cite{aron2018analytic}]\label{Prop-fibers-GPs} 
In $\M(\A_u(B_{\ell_p}))$ we have:
\begin{enumerate}[\upshape (a)] 
\item \label{Disco en Mz y GP0} 
For any $z\in B_{\ell_p}$, there is a copy of a unit disk $\D$ in  $\M_z\cap \GP(\delta_0)$. 

\item \label{Bola en M0 y GP0} 
There is a copy of $B_{\ell_p}$ in $\M_0\cap \GP(\delta_0)$.

\end{enumerate}  
\end{proposition}

We now prove that $\M_0$ intersects lots of Gleason parts. First we introduce the following notation. For a point $\varphi$ and a set $E$ in $\M(\A)$, 
we write $\varphi\in w^*-\acc E$ to mean that  $\varphi$ is a $w^*$-accumulation point of $E$.

\begin{theorem}  
\label{BetaNenM0}
The fiber $\M_0$ in $\M(\A_u(B_{\ell_p}))$ contains a set of cardinal $2^{\mathfrak c}$ such that any two elements of this set belong to different Gleason parts.
\end{theorem}

\begin{proof} We first give a sequence $\{\psi_k\}_k$ in $\M_0$ by taking
$\{\Nat_k\}_{k\in \Nat}$ a countable partition of 
$\Nat$ such that each $\Nat_k$ is infinite. For each  $k$, take $\psi_k \in w^*-\acc\{\delta_{e_n}\colon n\in\Nat_k\}$  that clearly belongs to $\M_0$. We claim that $\GP(\psi_k)\not=\GP(\psi_l)$ for any $k\ne l$. To see this, pick  $m\ge p$ and consider, for each $k$, the $m$-homogeneous polynomial $P_k\in \A_u(B_{\ell_p})$, defined by $P_k(x)=\sum_{n\in \Nat_k}x_n^m$ satisfying that $\|P_k\|=1$.  
As $\psi_k(P_k)=1$ and $\psi_l(P_k)=0$ for any $l\ne k$, $\rho(\psi_k,\psi_l)=1$, and the claim follows.
\\
Now, let us show that $\{\psi_k\}_k$ is an interpolating sequence. Take $m\ge p$,  $(c_k)_k\in \barB_{\ell_\infty}$ and define the $m$-homogeneous polynomial $Q\in \A_u(B_{\ell_p})$ as $Q(x)=\sum_{k\in \Nat}c_k \sum_{n\in \Nat_k} x_n^m$. Since for each $n\in \Nat_k$, $\delta_{e_n}(Q)=c_k$ we have $\psi_k(Q)=c_k$. Hence, the conclusion holds.
\\
This means that the set $\mathcal C=\overline{\{\psi_k\}_k}^{w^*}$, which is included in $\M_0$, is homeomorphic to  $\beta(\Nat)$, the Stone--Cech compactification of $\N$, that has cardinal $2^{\mathfrak c}$. As any element of  $\mathcal C$ belongs to $w^*-\acc\{\delta_{e_n}\colon n\in\Nat\}$, for $\varphi\ne \psi\in \mathcal C$, we can find disjoint subsets  $\N_\varphi\not=\N_\psi\subset\N$ such that $\varphi\in w^*-\text{ac}\{\delta_{e_n}\colon n\in\Nat_\varphi\}$ and $\psi\in w^*-\text{ac}\{\delta_{e_n}\colon n\in\Nat_\psi\}$. Finally, proceeding as in the first part of the proof we obtain that $\GP(\varphi)\not=\GP(\psi)$.
\end{proof}

We can extend this result to all the fibers over points inside the ball in the case that $p$ is an integer. We proceed in two steps.

\begin{lemma}\label{GLenlpnatural} 
Let $p\in \Nat$, $p\ge 2$ and let $z=(z_n)_n\in B_{\ell_p}$, $z_n\ge 0$ for all $n$, and  $(a_n)_n$ a sequence such that $a_n>0$ and $\|z + a_ne_n\|_p=1$. Then, for $\M(\A_u(B_{\ell_p}))$, any $\varphi, \psi\in w^*-\acc\{\delta_{z + a_ne_n}\colon n\in\Nat\}$ such that $\varphi\ne \psi$,
satisfy that $\varphi, \psi\in \M_z$ and  $\GP(\varphi)\not=\GP(\psi)$.
\end{lemma}
\begin{proof} It is clear that any element $\phi \in w^*-\acc\{\delta_{z + a_ne_n}\}$ belongs to  $\M_z$. Without loss of generality, we assume that there are infinite disjoint sets $\Nat_1$ and $\Nat_2$ partitioning $\Nat$ such that $\varphi\in w^*-\acc\{\delta_{z + a_ne_n}\colon n\in\Nat_1\}$ and $\psi\in w^*-\acc\{\delta_{z + a_ne_n}\colon n\in\Nat_2\}$. 
Now, we define a sequence of polynomials $(Q_k)$ satisfying the hypotheses of Lemma~\ref{Lema-util}. If $\Nat_2 = \{m_l\}_{l=1}^\infty$, for each  $k$  take $Q_k(x)=\sum_{n\in \Nat_1} x_n^p + \sum_{l=1}^k x_{m_l}^p$. Since $Q_k\in \P(^p \ell_p)$, $Q_k|_{B_{\ell_p}}\in\A$. Moreover,  $\|Q_k\|\le 1$ for every $k$.
\\
In order to calculate $\varphi(Q_k)$ we consider $j\in\Nat_1$. Noting that $(z_j+a_j)^p-z_j^p=1-\|z\|_p^p$ we have
$$
\begin{array}{rl}
\delta_{z+a_je_j}(Q_k) & =\sum\limits_{n\in\Nat_1 } z_n^p  - z_j^p+  (z_j+a_j)^p +  \sum_{l=1}^k z_{m_l}^p\\
& =\sum\limits_{ n\in\Nat_1} z_n^p  + 1 - \|z\|_p^p + 
\sum_{l=1}^k z_{m_l}^p =  1- \sum\limits_{l>k} z_{m_l}^p.
\end{array}
$$
Hence, $\varphi(Q_k) = 1- \sum\limits_{l>k} z_{m_l}^p$ and $\lim_{k\to\infty} \varphi(Q_k)=1$.

Now, to obtain the value of $\psi(Q_k)$,  we consider  $j\in\Nat_2$, $j>m_k$. We have
$$
\delta_{z+a_je_j}(Q_k)=  \sum_{n\in\Nat_1} z_n^p +  \sum_{l=1}^k z_{m_l}^p,
$$
and then, 
\[\psi(Q_k)=\sum_{n\in\Nat_1} z_n^p +  \sum_{l=1}^k z_{m_l}^p.
\] 
Thus $\lim_{k\to\infty}\psi(Q_k)=\|z\|_p^p<1$. Finally, the conclusion follows by Lemma~\ref{Lema-util}.
\end{proof}

\begin{proposition} \label{prop: p entero}
Let $p\in\N$, $p\ge 2$. Then for each  $z\in B_{\ell_p}$, the fiber $\M_z$ in $\M(\A_u(B_{\ell_p}))$ contains a set of cardinal $2^{\mathfrak c}$ such that each two elements of this set belong to different Gleason parts.
\end{proposition}

\begin{proof} 
It is enough to consider $z=(z_n)$ with $z_n\ge 0$ for all $n$. Indeed, if this is not the case, let $\theta_n$ be such that $e^{i\theta_n}z_n=|z_n|$ and define $\Phi\colon B_{\ell_p}\to B_{\ell_p}$ to be the automorphism  given by $\Phi(x)=(e^{i\theta_n}x_n)$ which satisfies the hypotheses of \cite[Prop.~1.6]{aron2020gleason}. Then it induces a mapping $\Lambda_\Phi\colon \M(\A_u(B_{\ell_p}))\to \M(\A_u(B_{\ell_p}))$ which is an isometry for the Gleason metric that sends $\M_z$ onto $\M_{(|z_n|)}$. 

As in Theorem~\ref{BetaNenM0} we obtain an interpolating sequence $\{\psi_k\}_k$ in $\M_z$ as follows. 
For each $n$ take $a_n>0$ satisfying $\|z + a_ne_n\|_p=1$. Now, let $\{\Nat_k\}_{k\in \Nat}$ be a countable  partition of $\Nat$ such that each $\Nat_k$ is infinite and choose $\psi_k$ in $w^*-\acc\{\delta_{z +
a_ne_n}\colon n\in\Nat_k\}$. It is clear that $\psi_k\in \M_z$ for all $k$ and, by Lemma~\ref{GLenlpnatural}, $\GP(\psi_k)\not=\GP(\psi_l)$, for every $k\neq l$.\\ 
Let $r=1-\|z\|_p^p$ and fix $(c_k)_k\in \barB_{\ell_\infty}$. We define  $P\in \A$ to be 
the polynomial given by 
$$
P(x)=\sum_{k\in \Nat} \frac{c_k}{r} \sum_{n\in \Nat_k} (x_n-z_n)^p.
$$ 
For each  $j\in \Nat_k$ it holds that $$\delta_{z + a_j e_j}(P)= \frac{c_k}{r} a_j^p.$$ 
Since we have chosen $a_j$ such that $a_j=\big(1- \|z\|_p^p + z_j\big)^{1/p}-z_j$, then $a_j^p\underset{j\to\infty}{\to} 1- \|z\|^p=r$ and 
$$
\psi_k(P)= w^*-\lim_{\Nat_k} \delta_{ z + a_je_{j}}(P)= c_k.
$$
Hence, $\{\psi_k\}_k$ is interpolating and the conclusion follows as in Theorem~\ref{BetaNenM0}.
\end{proof}

\begin{remark}\label{rmk:automorfismo beta}
    For the particular case $p=2$,  Theorem~\ref{BetaNenM0} gives an easier way to obtain the result of the previous proposition. Indeed, we know by \cite[Ex.~1.8]{aron2020gleason} that for each $z\in B_{\ell_2}$ there is an automorphism $\beta_z\colon B_{\ell_2}\to B_{\ell_2}$ which produces an onto isometry $\Lambda_{\beta_z}\colon \M(\A_u(B_{\ell_2}))\to \M(\A_u(B_{\ell_2}))$ satisfying $\Lambda_{\beta_z}(\M_0)=\M_z$. Hence, all the Gleason structure of the fiber $\M_0$ replicates in $\M_z$.
\end{remark}
%%%%%%%%%%%%%%%%%%%%%%%%%%%%%%%%%%%%%%%%%%%%%%%%%%%%%%
\section{Fibers and Gleason parts of $\M(\A_u(B_{\ell_1}))$.}
\label{Sec:ele1}
%%%%%%%%%%%%%%%%%%%%%%%%%%%%%%%%%%%%%%%%%%%%%%%%%%%%%%

Now we focus on the case $p=1$ which deserves a particular approach due to relevant differences with the previous situation:  $\ell_1$ is not reflexive (even is not strictly convex) and $\ell_1$ is not symmetrically regular. We refer the reader to \cite{aron1996regularity} for the definition of symmetric regularity and also to \cite{aron1991spectra} where this property was used to understand the analytical structure of the spectrum of different function algebras.
We begin by recalling what is already known  about the fibers of $\M(\A_u(B_{\ell_1}))$:

\begin{proposition}\label{prop:ell1-fibers}
    
 In $\M(\A_u(B_{\ell_1}))$ we have:
\begin{enumerate}[\upshape (a)] 
    \item \cite[Prop.~3.1]{aron2016cluster} If $z\in S_{\ell_1}$ then $\M_{z}=\{\delta_{z}\}$. 
    \item \cite[Thm.~3.1]{ aron2018analytic} If $z\in B_{\ell_1''}$, the complex disk $\D$ can be analytically injected in $\M_{z}$.
    \item \cite[Prop.~3.2]{aron2016cluster} If $z\in B_{\ell_1''}$, $\M_{z}$ contains a set of cardinal $2^{\mathfrak c}$.
     \item \cite[Prop.~3.3]{aron2016cluster} $S_{\ell_1''}$
     contains a set $\mathcal K$ of cardinal $2^{\mathfrak c}$ such that each fiber over an element of $\mathcal K$ has at least cardinal $\mathfrak{c}$.
\end{enumerate}
\end{proposition} 

From the results listed above, it is not clear if every fiber over a point $z\in S_{\ell_1''}\setminus S_{\ell_1}$ is nontrivial (i.e. not singleton). 
We are now in a position to prove that this is so by means of arguments lying on a close look at the structure of the points $z$. Recall that $c_0$ is an $M$-ideal in $\ell_\infty$, so there is a decomposition $\ell_1''=\ell_1\oplus_1 c_0^\perp$. Analyzing the proof of the previous statement (d), we see that the set $\mathcal K$ constructed in \cite[Prop.~3.3]{aron2016cluster} lies inside $c_0^\perp$. We can also reach fibers over some points of  $S_{\ell_1''}\setminus S_{c_0^\perp}$ with a similar strategy. But even then, we would not cover all cases. To complete the description, we use a combination of two different approaches: first, we identify the points $z$ for which a variation of the construction from \cite{aron2016cluster} is allowed, and then we develop a new scheme for the remaining points.

We state here the general result whose proof is obtained by merging the subsequent Propositions \ref{prop:real extreme point} and \ref{prop:fibers_thick}.

\begin{theorem}\label{thm:fibras de ele 1}
    For  $\A_u(B_{\ell_1})$, the fiber $\M_z$ over any $z\in S_{\ell_1''}\setminus S_{\ell_1}$ satisfies $\Card(\M_z)\ge \mathfrak{c}$.
\end{theorem}

On the way to our goal, it is useful to have a handy
description of the real extreme points of $\overline B_{\ell_1''}$, and, consequently, of $\overline B_{c_0^\perp}$.

\begin{lemma}
The set of real extreme points of $\overline B_{\ell_1''}$ is
$$\Ext_\R(\overline B_{\ell_1''})=\{\lambda z: z\in \overline{\{e_n\}}^{w^*}, \lambda\in\C,  |\lambda|=1\}.$$ In particular, the set of real extreme points of $\overline B_{c_0^\perp}$  is
$$\Ext_\R(\overline B_{c_0^\perp})=\{\lambda z: z\in w^*-\acc\{e_n\}, \lambda\in\C,  |\lambda|=1\}.$$   
\end{lemma}
\begin{proof}
  The identification $\ell_1''=\ell_\infty'=C(\beta\N)'$ gives us a description of the real extreme points of its unit ball: $\Ext_\R(\overline B_{\ell_1''})=\{\lambda \xi_s: s\in\beta\N, \lambda\in\C,  |\lambda|=1\}$, where $\xi_s(a)=a(s)$ for every $a\in\ell_\infty=C(\beta\N)$. Notice that if $s=n\in\N$, then $\xi_n=e_n\in\ell_1$. Now, if $s\in\beta\N\setminus\N$, there is a net $(n_\alpha)\subset\N$ converging to $s$ in the topology of $\beta\N$. This is equivalent to say that the net $\{e_{n_\alpha}\}\subset \ell_1''$ is $w^*$-convergent to $\xi_s$. In addition, in this case, $\xi_s$ belongs to $c_0^\perp$. As a byproduct of this discussion along with the identity $\Ext_\R(\overline B_{\ell_1''})=\Ext_\R(\overline B_{\ell})\cup \Ext_\R(\overline B_{c_0^\perp})$, the second statement follows.
\end{proof}

Now we are ready to extend the referred argument from \cite[Prop.~3.3]{aron2016cluster} to the fibers over points of $S_{\ell_1''}$ whose component in $c_0^\perp$ is a real extreme point of the ball of its norm.

\begin{proposition}\label{prop:real extreme point}
    If $z=x+w\in S_{\ell_1''}$ with $x\in\ell_1$ and $w\in c_0^\perp\setminus\{0\}$ such that $\frac{w}{\|w\|}$ is a real extreme point of $\overline B_{c_0^\perp}$ then, the fiber $\M_z$ of $\M(\A_u(B_{\ell_1}))$ satisfies $\Card(\M_z)\ge \mathfrak{c}$.
\end{proposition}

\begin{proof}
    It is clear that, for every $z\in S_{\ell_1''}$ and every $\lambda\in\C$ with $|\lambda|=1$, $\M_z$, the fiber over $z$, is homeomorphic (and  Gleason isometric) to $\M_{\lambda z}$, the fiber over $\lambda z$. Hence, in this proof it is enough to consider $z=x+w\in S_{\ell_1''}$ with $x\in\ell_1$ and $w\in c_0^\perp$ such that $\frac{w}{\|w\|}\in w^*-\acc\{e_n\}$. Choose a net ${\{e_{n_\alpha}\}}$ satisfying $ e_{n_\alpha}\overset{w^*}{\to}\frac{w}{\|w\|}$ and take $\varphi$ to be a $w^*$-limit point of $\{\delta_{x+\|w\|e_{n_\alpha}}\}\subset \M(\A_u(B_{\ell_1}))$. 
    Note that $\varphi\in\M_z$ and let us now see that $\delta_z\not=\varphi$. Indeed, consider the polynomial $P$ given by $P(y)=\sum_{j=1}^\infty (y_j-x_j)^2$. On the one hand, since $P(x+\|w\|e_{n_\alpha})=\|w\|^2$, for every $\alpha$, we have $\varphi(P)=\|w\|^2$. On the other hand, developing the polynomial $P$ in its homogeneous parts and arguing as in the proof of  \cite[Prop.~2.5]{aron2016cluster} we derive that $\delta_z(P)=0$.
    Finally, to derive $\Card(\M_z)\ge \mathfrak{c}$ we proceed as in the end of the proof of \cite[Prop.~3.3]{aron2016cluster}, which appeals to the connection of the   {\it cluster set} of $P$ at $z\in
\barB_{\ell_1''}$ (i.e. the set consisting of all limits of values of $P$ along nets in
$B_{\ell_1}$ $w^*$-converging  to $z$).
    \end{proof}

Our next step is to show that the conclusion of the above proposition remains valid for points $z=x+w\in S_{\ell_1''}$ with $x\in\ell_1$ and $w\in c_0^\perp$ such that $\frac{w}{\|w\|}$ is not a real extreme point of $\overline B_{c_0^\perp}$. This would complete the description of the fibers over points in $ S_{\ell_1''}\setminus S_{\ell_1}$ thus finishing the proof of Theorem \ref{thm:fibras de ele 1}.

Before going on, we recall some concepts and known results about a relative algebra of analytic mappings. Given $X$ a complex Banach space, $\H_b(X)$ denotes the Fr\'echet algebra of entire functions bounded on bounded subsets of $X$, endowed with the topology of uniform convergence on these bounded sets. Several properties of the maximal ideal space $\M(\H_b(X))$ are studied in \cite{aron2016cluster} where,  following \cite{aron1991spectra}, the \textit{radius function} of $\varphi \in \M(\H_b(X))$
is defined by
$$
R(\varphi)=\inf\big\{r>0\,\colon \, |\varphi(f)|\le \|f\|_{rB_X} \text{ for every } f \in \H_b(X)\,\big\}
$$ where $\|f\|_{rB_X}=\sup\{|f(x)|\colon \|x\|<r\}$. In particular, $\|\varphi|_{X'}\| \le R(\varphi)$ \cite[Lem.~3.2]{aron1991spectra} and $R(\delta_z)=\|z\|$ for all $z\in\barB_{X''}$ \cite[Lem.~3.1]{aron1991spectra}.

The spectra of $\H_b(X)$ and $\A_u(B_X)$ are related via the following onto homeomorphism, see  
\cite[Lem.~1.2]{aron2016cluster} or \cite[Thm.~12.1]{aron1991spectra}:  
\begin{equation}\label{homeomorfismo}
\Phi\colon \M(\A_u(B_X))\to \big\{\varphi\in \mathcal{M}(\H_b(X)),\, R(\varphi)\leq 1\big\},  
\end{equation}
defined by $\Phi(\varphi)(f)=\varphi(f|_{B_X}),$ for  
$\varphi\in \M(\A_u(B_X))$ and $f\in \H_b(X)$. This identification will allow us to transfer results from one to another spectrum and will be used without further mention. 

Another important tool within the spectrum of $\H_b(X)$ (that we will make use of) is the convolution product $\ast\colon \mathcal{M}(\H_b(X))\times \mathcal{M}(\H_b(X))\to \mathcal{M}(\H_b(X))$, defined in \cite[Sect. 6]{aron1991spectra} as follows: for $\varphi, \psi\in \M(\H_b(X))$ and $f\in\H_b(X)$,
$$
\varphi\ast\psi(f)=\varphi(x\mapsto \psi(\tau_x(f))),
$$ where for each $x\in X$, $\tau_x(f)\in\H_b(X)$ is given by $\tau_x(f)(y)=f(x+y)$ for all $y \in X$.

This operation satisfies  two relevant properties:
\begin{equation}\label{radioacotacion}
R(\varphi\ast \psi)\leq R(\varphi)+R(\psi),   
\end{equation}
and given $z,w\in X''$ such that  $\varphi\in \mathcal{M}_z(\mathcal {H}_b(B_X))$ and $\psi\in \mathcal{M}_w(\H_b(B_X))$, then
\begin{equation}\label{sumafibras}
\varphi\ast \psi \in \mathcal{M}_{z+w}(\H_b(B_X)).
\end{equation}
(See \cite[Lem.~6.2]{aron1991spectra} and \cite[Lem.~6.4]{aron1991spectra} respectively.)

Finally, we will appeal to a simple composition action, 
coming from a family of one-parameter operators, $(t\Id)_t$, where $\Id$ is the identity map and $t\in \C$, see \cite[Sect.~5]{aron1991spectra}.
For each $t\in \C$,   $\Upsilon_t\colon \H^*_b(X) \to \H^*_b(X)$, is defined by 
$\Upsilon_t(\varphi)(f)= \varphi(f\circ t\Id)$. For simplicity, we denote $\varphi^t:=\Upsilon_t(\varphi)\in \H^*_b(X)$. Notice that, $\varphi^0 =\delta_0$ and $\varphi^1 = \varphi$, also, for any $m$-homogeneous polynomial $P$, $\varphi^t(P)=t^mP$. 

With the aid of these tools, we show that any nontrivial fiber in $\M(\A_u(B_X))$ has cardinality at least $\mathfrak{c}$.

\begin{lemma}\label{lem:card phi-t}  Let $X$ be a complex Banach space. For any $\varphi, \psi\in \M(\H_b(X))$, $\varphi\ne \psi$, 
$$
\Card\left(\{\varphi^t \ast \psi^{1-t}\colon t\in [0,1]\}\right) \ge \mathfrak{c}.
$$
\end{lemma}

\begin{proof} Since $\varphi\ne \psi$ and the analytic polynomials are dense in $\H_b(X)$, there exists  a continuous $m$-homogeneous  polynomial $P$ such that 
$\varphi(P)\ne \psi(P)$.
Now, define $q\colon \C\to \C$, by $q(t)=\varphi^t\ast \psi^{1-t}(P)$, which results a polynomial of degree non greater than $m$.  
The conclusion follows by noting that $q(0)=\varphi^0\ast \psi^{1}(P)=\delta_0\ast \psi(P)= \psi(P)\ne q(1)=\varphi \ast \delta_0(P)=\varphi(P)$.
\end{proof}

\begin{lemma}\label{lem:card Mz +c}  Let $X$ be a complex Banach space and $z\in \barB_{X''}$. If there exist $\varphi\not=\psi\in\M_z$ in $\M(\A_u(B_X))$, then 
$\Card\left(\M_z\right) \ge \mathfrak{c}.$
\end{lemma}

\begin{proof} By the identification given in \eqref{homeomorfismo}, both homomorphisms $\varphi, \psi \in \M(\A_u(B_X))$ may be regarded as elements in $\M(\H_b(X))$ with $R(\varphi), R(\psi)\le 1$. 

Let us consider the set $\mathfrak{C}=\{\varphi^t \ast \psi^{1-t}\colon t\in [0,1]\}$ in $\M(\H_b(X))$. By~\cite[Lem.~5.4]{aron1991spectra} $\varphi^t \in \M_{tz}(\H_b(X))$ and $\psi^{1-t}\in \M_{(1-t)z}(\H_b(X))$. Now, \eqref{sumafibras} gives that $\varphi^t \ast \psi^{1-t}\in \M_{z}(\H_b(X))$ for all $0\le t\le 1$. Again by \cite[Lem.~5.4]{aron1991spectra}, $R(\varphi^t)= tR(\varphi)$ and  $R(\psi^{1-t})= (1-t)R(\psi)$, and an application of \eqref{radioacotacion}  yields that $R(\varphi^t \ast \psi^{1-t})\le 1$.

Finally, appealing to \eqref{homeomorfismo}, we see that $\mathfrak{C}\subset \M_z$ in $\M(\A_u(B_X))$ and the result follows by Lemma~\ref{lem:card phi-t}.
\end{proof}

Now we are ready to present the announced result which completes the proof of Theorem~\ref{thm:fibras de ele 1}.

\begin{proposition}\label{prop:fibers_thick}
    Let $z=x+w\in S_{\ell_1''}$ with $x\in\ell_1$ and $w\in c_0^\perp\setminus\{0\}$ such that $\frac{w}{\|w\|}$ is not a real extreme point of $\overline B_{c_0^\perp}$. Then, for  $\A_u(B_{\ell_1})$, the fiber $\M_z$ satisfies $\Card(\M_z)\ge \mathfrak{c}$.
\end{proposition}

\begin{proof} 
Since $w$ is not a real extreme point of $\|w\|\overline B_{c_0^\perp}$, there exist $w_1,w_2\in c_0^\perp$ with $\|w\|=\|w_1\|=\|w_2\|$ and $w_1\neq w_2$ such that $w=\frac{w_1+w_2}{2}$.
Note that $w_2$ cannot be a multiple of 
$w_1$. Indeed, if  $w_2=\lambda w_1$, then $\|w\|=\left\|\frac{w_1+\lambda w_1}{2}\right\|=\left| \frac{1+\lambda}{2}\right|\|w\|$.
Hence $\lambda=1$, which is a contradiction.  

Now we take $z_1=x+ \frac{w_1}{2}$ and $z_2=\frac{w_2}{2}$. Since $\ell_1''=\ell_1\oplus_1 c_0^\perp$, we have $\|z_1\|=\|x\|+\frac{\|w\|}{2}$ and $\|z_2\|=\frac{\|w\|}{2}$.
This, clearly implies that $z=z_1+z_2$ with $1=\|z\|=\|x\|+\|w\|=\|z_1\|+\|z_2\|$.

We define the homomorphisms  $\varphi=\delta_{z_1}\ast  \delta_{z_2}$ and $\psi=\delta_{z_2}\ast  \delta_{z_1}$ and make use of \cite[Thm.~6.11]{aron1991spectra} where it was proved, in a general context, that $\delta_{z_1}\ast  \delta_{z_2}=\delta_{z_2}\ast  \delta_{z_1}$ if and only $\delta_{z_1}\ast  \delta_{z_2}= \delta_{z_1+z_2}$. Also, for $X=\ell_1$ something else can be said appealing to \cite[Cor.~2.5]{carando2023homomorphisms}:  since $w_2$ is not a multiple of $w_1$ we have that $\varphi\neq  \psi $ and both homomorphisms differ from $\delta_z$. As $\varphi$ is the convolution of $\delta_{z_1}$ and $\delta_{z_2}$, by \eqref{radioacotacion} we have $R(\varphi)\leq R(\delta_{z_1}) +R(\delta_{z_1})= \|z_1\| +\|z_2\|= 1$. Analogously, $R(\psi)\leq 1$. Also, by \eqref{sumafibras}, both $\varphi$ and $\psi $ belong to $\M_z(\H_b(\ell_1))$. Hence, due to \eqref{homeomorfismo} along with Lemma~\ref{lem:card Mz +c} we obtain the conclusion. 
\end{proof}

Once we have distinguished between singleton and non-singleton fibers, we focus on how they interact with Gleason parts. Although we believe that the path is paved, our progress in this regard has been limited.

In Proposition \ref{prop:GP_complex_extrem} we have seen that if $z\in\Ext_{\C}(\overline B_{X''})$ and $\varphi\in\M_z$ then $\GP(\varphi)$ cannot intersect any other fiber. In light of this result, it will be useful to have a description of $\Ext_{\C}(\overline B_{\ell_1''})$. We thank our friend Tomás Rodríguez for drawing our attention to the following fact.

\begin{remark}
  The space $\ell_\infty'$ is isometrically isomorphic to $L_1(\mu)$, for a certain measure $\mu$. 

  \rm Indeed, this is a consequence of the duality between abstract $M$- and $L$-spaces (see \cite{lacey2012isometric} for the definitions) combined with a result by Kakutani. Since we could not find a complete reference for these arguments in the case of general complex Banach lattices, we include an exposition of its derivation from the real case. In this reasoning, we distinguish real and complex spaces in a standard way. It is clear that $\ell_\infty(\R)$ is an abstract $M$-space and thus \cite[Thm. 7, pg. 25]{lacey2012isometric} implies that $\ell_\infty(\R)'$ is an abstract $L$-space. Now, due to Kakutani's theorem \cite[Thm. 3, pg. 135]{lacey2012isometric} we know that $\ell_\infty(\R)'$ is isometrically isomorphic to $L_1(\mu, \R)$, for a certain measure $\mu$.

Now, we use several facts about  Bochnack and Taylor's complexifications from \cite{munoz1999complexifications}. As in this referred article, we denote by $\widetilde X$ the complexification of a real space $X$ and  $\|\cdot\|_B$ and $\|\cdot\|_T$ for Bochnack and Taylor's norms in the complexified space.

$$
\ell_\infty(\C)'=(\widetilde{\ell_\infty(\R)}, \|\cdot\|_T)'=(\widetilde{\ell_\infty(\R)'}, \|\cdot\|_B)\cong (\widetilde{L_1(\mu, \R)},\|\cdot\|_B) = L_1(\mu, \C).
$$
\end{remark}

Now, from \cite[Th. 4.2]{thorp1967strong}, which states that every point on the unit sphere of a $L_1(\mu)$ space is a complex extreme point, combined with Proposition \ref{prop:GP_complex_extrem} we derive the following:

\begin{corollary}
   The set of complex extreme points of the unit ball of $\ell_1''$ is $\Ext_{\C}(\overline B_{\ell_1''})= S_{\ell_1''}$. Consequently, for any $z\in S_{\ell_1''}$ and $\varphi\in\M_z$, the Gleason part $\GP(\varphi)$ cannot intersect any other fiber.
\end{corollary}

The existence of homomorphisms from the same fiber belonging to different Gleason parts seems to be tougher to establish in $\M(\A_u(B_{\ell_1}))$ than in $\M(\A_u(B_{\ell_p}))$ ($1<p<\infty$). A close inspection to the proof of Proposition \ref{prop:real extreme point} yields our only result about this question:

\begin{corollary}
    If $w$ is a real extreme point of $\overline B_{c_0^\perp}$ then there is $\varphi\in\M_w$ such that $\GP(\varphi)\not=\GP(\delta_w)$.
\end{corollary}

Indeed, as in the referred proof (which in this case is not other than the proof of \cite[Prop.~3.3]{aron2016cluster}),  the polynomial $P(y)=\sum_{j=1}^\infty y_j^2$ has norm 1 and $\varphi(P)=1$ while $\delta_w(P)=0$.

%%%%%%%%%%%%%%%%%%%%%%%%%%%%%%%%%%%%%%%%%%%%%%%%%%%%%%
\section{Final comments and open questions}
\label{Sec:final}
%%%%%%%%%%%%%%%%%%%%%%%%%%%%%%%%%%%%%%%%%%%%%%%%%%%%%%

We devote this last section to listing some open problems arising from the work and to presenting some related comments or examples.

For any element $f$ in a uniform algebra $\A$ there  exists a strong boundary point $\varphi\in\M(\A)$ such that $\varphi(f)=1$; see for instance \cite[Thm.~7.21]{stout1971theory}. Consequently, it would be interesting to obtain a description of all the strong boundary points for the algebras $\A=\A_u(B_{\ell_p})$,  $1\le p<\infty$, that we have studied. In the case $1< p<\infty$, as we see in Proposition \ref{Prop-fibers-GPs1}, $\delta_z$ is a peak point and, hence, a strong boundary point for $\A$, for every $z\in S_{\ell_p}$. Also, in \cite[Prop.~3.6]{aron2020gleason} it is proved that all the strong boundary points for $\H^\infty(B_{c_0})$ are in  fibers over points of 
the distinguished boundary $\T^\N$ of $S_{\ell_\infty}$. So, one can be tempted to think that all the strong boundary points for $\A$ should be in fibers over \textit{boundary} points; that is, in fibers over points in $S_{\ell_p}$ (for $1< p<\infty$) or $S_{\ell_1''}$ (for $p=1$). The following example from \cite{aron1991spectra} shows that this is not the case, at least for $1<p<\infty$.

\begin{example} \cite[Lem.~11.2]{aron1991spectra}
Let $k\in\N$, $k\ge p>1$ and consider $f\colon\ell_p\to\C$ given by $f(x)=\sum_n (1-\frac1n)x_n^k$. It is clear that $f\in\A_u(B_{\ell_p})$ with $\|f\|=1$. Take a strong boundary point $\varphi\in\M(\A)$ such that $\varphi(f)=1$ which, by the argument from the referred lemma, belongs to $\M_0$.
    \end{example}

In \cite[Prop.~1.6]{aron2020gleason} it is shown that if we have an automorphism $\Phi:B_X\to B_X$ such that $\Phi$ and $\Phi^{-1}$ are uniformly continuous, then the associated mapping $\Lambda_\Phi\colon\M(\A_u(B_X))\to \M(\A_u(B_X))$, given by $\Lambda_\Phi(\varphi)(f)=\varphi(f\circ\Phi)$, is an onto isometry. In this case, it is easy to see that $\varphi$ is a strong boundary point if and only if $\Lambda_\Phi(\varphi)$ is a strong boundary point.  Therefore, the previous example along with what we have commented in Remark~\ref{rmk:automorfismo beta} allows us to conclude that all the fibers of $\M(\A_u(B_{\ell_2}))$ have strong boundary points. We do not know if the same holds for other values of $p$.

Changing our focus, what it is certainly always true is that the radius function on a strong boundary point should be 1. Indeed, if $\varphi\in\M(\A)$ is a strong boundary point, then there is a non-constant function $f\in\A$ such that $\varphi(f)=1$. Then, we cannot have $|\varphi(f)|\le \|f\|_{rB_{\ell_p}}$ for any $r<1$, meaning that $R(\varphi)=1$.

Continuing with strong boundary points, as already mentioned in Section~\ref{Sec:ele1}, a combination of \cite[Thm. 2.6]{acosta2007shilov} with \cite[Thm. 3.3 (2)]{choi2021boundaries} gives that for each $z\in S_{\ell_1}$, the evaluation $\delta_z$ is a peak point. Hence, $\delta_z$ is a strong boundary point. Also, 
given any $y=(y_n)\in S_{\ell_\infty}$ with $|y_n|<1$ for all $n$ there exists $z$, a real extreme point  of $\overline B_{\ell_1''}$, such that $\langle z,y\rangle =1$. It is evident that $z\in \Ext_{\R}(\overline B_{c_0^\perp})$. On the other hand, since $y\in\A_u(B_{\ell_1})$ there is a strong boundary point $\varphi\in\M(\A)$ such that $\varphi(y)=1$. So we can wonder if $\varphi=\delta_z$ or, at least, $\varphi\in \M_z$. 

Now we come to our first open question.

\textbf{Open problem 1.} Describe all the strong boundary points for $\A_u(B_{\ell_p})$,  ($1\le p<\infty$) or, at least, identify all the fibers containing strong boundary points.

 The existence of homomorphisms from different Gleason parts in the same fiber is almost unknown for $\M(\A_u(B_{\ell_1}))$. We have just proved it for the fiber $\M_w$, being $w$ a real extreme point of $\overline B_{c_0^\perp}$. Also, we know that $\M_z= \{\delta_z\}$ for $z\in S_{\ell_1}$. Thus, it would be interesting to solve it for the rest of the fibers.

\textbf{Open problem 2.} For $\M(\A_u(B_{\ell_1}))$, given $z\in \overline B_{\ell_1''}$ such that $z\not\in S_{\ell_1}$ and $z$ is not a real extreme point of $\overline B_{c_0^\perp}$, decide whether there exists $\varphi\in\M_z$ with $\GP(\varphi)\not=\GP(\delta_z)$.

Looking at the proof of the result referred to in Proposition~\ref{prop:ell1-fibers} (b) for the spectrum $\M(\A_u(B_{\ell_1}))$,  we derive that for each $z\in B_{\ell_1''}$ there is a copy of the complex disk contained in $\M_z\cap \GP(\delta_z)$, which, by Remark~\ref{ADLM basicos} coincides with $\M_z\cap \GP(\delta_0)$. Analogous results are stated in Proposition~\ref{Prop-fibers-GPs} for $\M(\A_u(B_{\ell_p}))$, $1<p<\infty$. Thus,  in our setting, all known cases of two homomorphisms (in fibers over interior points) belonging to the same Gleason part satisfy that the Gleason part is $\GP(\delta_0)$. This leads us to propose the next problem.

\textbf{Open problem 3.} For $\M(\A_u(B_{\ell_p}))$, $1\le p<\infty$, show the existence (or lack) of morphisms $\varphi\not= \psi$, lying in fibers over interior points, such that $\GP(\varphi)=\GP(\psi) \not= \GP(\delta_0)$.

Proposition~\ref{prop: p entero} has a restrictive unexpected hypothesis: $p$ should be integer. The reason is purely technical since we need $P(x)=\sum_n x_n^p$ to be a polynomial and, hence, an element of $\A_u(B_{\ell_p})$. It is interesting to know whether this hypothesis is necessary or not.

\textbf{Open problem 4.} For $\M(\A_u(B_{\ell_p}))$, $1< p<\infty$, if $p$ is not an integer, does the conclusion of Proposition~\ref{prop: p entero} hold?

%%%%%%%%%%%%%%%%%%%%%%%%%%%%%%%%%%%%%%%%%%%%%%%%%%%%%%
\section{Appendix}
%%%%%%%%%%%%%%%%%%%%%%%%%%%%%%%%%%%%%%%%%%%%%%%%%%%%%%

 In analyzing the interaction between fibers and Gleason parts within the spectrum of $\A_u(B_X)$, the set of complex extreme points of the closed unit ball of the bidual space has proven to be useful, as can be seen for example in Proposition~\ref{prop:GP_complex_extrem}. This fact motivates us, for the sake of completeness, to present in this appendix a description of the complex extreme points of the unit ball of $X \oplus_p Y$. Even though it was not necessary for our investigation, we believe that the following is of interest on its own.  

\begin{proposition}\label{prop:complex_extreme}
Given Banach spaces $X$ and $Y$ and $1\le p<\infty$ we have
$$
\Ext_{\C}(\overline B_{X\oplus_p Y})=\{x+y\in S_{X\oplus_p Y}\colon\ x\in\Ext_{\C}(\|x\|\overline B_{X})\ \ \text{\rm and }\ y\in\Ext_{\C}(\|y\|\overline B_{ Y})\}.
$$  
\end{proposition}
    
 To prove this, we make use of the next simple lemma, which is surely well-known. In its proof, included to cover all bases, we refer to the following implication valid for $a, b\in\mathbb{C}$:

\begin{equation}\label{complejos}
    |a|=\frac{|a+\lambda b|+|a-\lambda b|}{2}\ \ \text{ for all }\ \lambda \in \barD\ \Rightarrow\ b=0.
\end{equation}

Indeed, \eqref{complejos} holds since if $|a|=\frac{|a+\lambda b|+|a-\lambda b|}{2}$ then there exists $r_{\lambda}\in\R_{>0}$ such that $a+\lambda b=r_{\lambda}(a-\lambda b) $, except for at most one value of $\lambda=\lambda_0$ satisfying $a=\lambda_0 b$. Thus, $\frac{r_{\lambda}-1}{r_{\lambda}+1}a=\lambda b$. Choosing $\theta\in [0,2\pi)$ such that $a e^{i\theta}=|a|$ we obtain $\frac{r_{\lambda}-1}{r_{\lambda}+1}|a|=\lambda b e^{i\theta}$, which implies $\lambda b e^{i\theta}\in\R$ for all $\lambda\in\barD\setminus \{\lambda_0\}$. This can only be true if $b=0$.

\begin{lemma}\label{lema simple}
    Let $X$ be a Banach space, $1\le p<\infty$ and $x, u\in X$. The following are equivalent:
    \begin{enumerate}[\upshape (i)]
        \item \label{lem_extr_i} $\|x\|=\|x+\lambda u\|$ for all $\lambda \in \barD$. 
        \item \label{lem_extr_ii} $\|x\|^p=\frac{\|x+\lambda u\|^p+\|x-\lambda u\|^p}{2}$ for all $\lambda \in \barD$.
    \end{enumerate}
\end{lemma}

\begin{proof}
  The implication \eqref{lem_extr_i} $\Rightarrow$ \eqref{lem_extr_ii} is obvious. Now, suppose that \eqref{lem_extr_ii} holds and let $x'\in S_{X'}$ such that $x'(x)=\|x\|$. Since
\begin{eqnarray*}
   \|x\| &=& x'(x)\le \frac{|x'(x+\lambda u)|+|x'(x-\lambda u)|}{2}\le\frac{\|x+\lambda u\|+\|x-\lambda u\|}{2}\\
   &\le & \left(\frac{\|x+\lambda u\|^p+\|x-\lambda u\|^p}{2}\right)^{1/p}=\|x\|  
\end{eqnarray*}
 
we obtain that $|x'(x+\lambda u)|=\|x+\lambda u\|$ for all $\lambda \in \barD$ and also
$$
x'(x)=|x'(x)|= \frac{|x'(x)+\lambda x'(u)|+|x'(x)-\lambda x'(u)|}{2}\quad\textrm{ for all }\lambda \in \barD.
$$ 
This implies, by \eqref{complejos}, that $x'(u)=0$. 

In consequence, $\|x\|=|x'(x)|=|x'(x+\lambda u)|=\|x+\lambda u\|$  for all $\lambda \in \barD$, and  \eqref{lem_extr_i} holds.
\end{proof}

\begin{proof}[Proof of Proposition~\ref{prop:complex_extreme}] 
   If $z=x+y\in  \Ext_{\C}(\overline B_{X\oplus_p Y})$ with $x\in X$ and $y\in Y$ then $1=\|z\|^p=\|x\|^p+\|y\|^p$. Suppose that $\|x+\lambda u\|\le \|x\|$ for some $u\in X$ and all $\lambda \in \barD$. Then $\|z+\lambda u\|\le 1$ for all $\lambda \in \barD$ which implies $u=0$. This proves that $x\in\Ext_{\C}(\|x\|\overline B_{X})$. Analogously, it is shown that $y\in\Ext_{\C}(\|y\|\overline B_{Y})$.

Conversely, given $x\in\Ext_{\C}(\|x\|\overline B_{X})$ and $y\in\Ext_{\C}(\|y\|\overline B_{Y})$ such that  $\|x\|^p+\|y\|^p=1$, let us see that $z=x+y$  belongs to $\Ext_{\C}(\overline B_{X\oplus_p Y})$. Suppose that $w\in X\oplus_p Y$, $w=u +v$ ($u\in X$ and $v\in Y$) satisfies $\|z+\lambda w\|\le 1$ for all $\lambda \in \barD$. Then, $\|x+\lambda u\|^p+\|y+\lambda v\|^p\le 1$ for all $\lambda \in \barD$. By a simple triangle inequality argument this leads to the following identities, valid for every $\lambda \in \barD$:
$$
\|x\|^p=\frac{\|x+\lambda u\|^p+\|x-\lambda u\|^p}{2};\qquad \|y\|^p=\frac{\|y+\lambda v\|^p+\|y-\lambda v\|^p}{2}.
$$ By Lemma \ref{lema simple}, we derive that $\|x\|=\|x+\lambda u\|$  and $\|y\|=\|y+\lambda v\|$ for all $\lambda \in \barD$ implying, by our assumption, that $u=v=0$, which finishes the proof.
\end{proof}

Observe that the relationship for real extreme points coincides with the one in Proposition~\ref{prop:complex_extreme} for $1<p<\infty$ but is different for $p=1$. Indeed,
$$\Ext_{\R}(\overline B_{X\oplus_1 Y})=\Ext_{\R}(\overline B_{X})\cup\Ext_{\R}(\overline B_{ Y}),$$ while, for $1<p<\infty$,
$$\Ext_{\R}(\overline B_{X\oplus_p Y})=\{x+y\in S_{X\oplus_p Y}\colon\ x\in\Ext_{\R}(\|x\|\overline B_{X})\ \ \text{\rm and }\ y\in\Ext_{\R}(\|y\|\overline B_{ Y})\}.$$

%%%%%%%%%%%%%%%%%%%%%%%%%%%%%%%%%%%%%%%%%%%%%%%%%%%%%%
%%%%%%%%%%%%%%%%%%%%%%%%%%%%%%%%%%%%%%%%%%%%%%%%%%%%%%

\noindent \textbf{Acknowledgments:} We are very grateful for the generosity of several friends and colleagues who certainly helped us through this research. First of all, we would like to thank  Richard Aron since this article grew out of discussions we had while completing \cite{aron2020gleason}. We thank him for his valuable comments, questions, and fruitful conversations. We also want to thank Tomás Rodríguez for his suggestions after a careful reading of an earlier version of the article, which led to an improvement of Section \ref{Sec:ele1}. Finally, we thank Andreas Defant and Mietek Masty{\l}o for helpful interactions. 

\bibliographystyle{abbrv}

\end{document}